\documentclass[11pt]{article}

\usepackage[margin=1in]{geometry}
\usepackage{amsmath,amssymb,amsthm,mathtools}
\usepackage{enumitem}
\usepackage{hyperref}
\usepackage{microtype}
\newcommand{\pmodtight}[1]{\,(\operatorname{mod} #1)}
\allowdisplaybreaks[2]
\AtBeginDocument{%
  \setlength{\abovedisplayskip}{7pt plus 3pt minus 2pt}%
  \setlength{\belowdisplayskip}{7pt plus 3pt minus 2pt}%
  \setlength{\abovedisplayshortskip}{4pt plus 2pt minus 2pt}%
  \setlength{\belowdisplayshortskip}{5pt plus 2pt minus 2pt}%
}
\makeatletter
\renewcommand\section{\@startsection {section}{1}{\z@}%
  {-2.6ex \@plus -0.8ex \@minus -.2ex}%
  {1.2ex \@plus .2ex}%
  {\normalfont\Large\bfseries}}
\def\thm@space@setup{%
  \thm@preskip=6pt plus 2pt minus 2pt
  \thm@postskip=6pt plus 2pt minus 2pt}
\makeatother

\hypersetup{
  colorlinks=true,
  linkcolor=blue,
  citecolor=blue,
  urlcolor=blue
}

\newtheorem{theorem}{Theorem}[section]

\newtheorem{lemma}[theorem]{Lemma}

\newcommand{\ord}{\operatorname{ord}}

\newcommand{\calD}{\mathcal D}

\newcommand{\calS}{\mathcal S}

\newcommand{\calE}{\mathcal E}

\title{An improved lower bound for odd integers not of the form $p+2^a+2^b$}
\author{Yuchen Ding, Yu-Chen Sun, and Lilu Zhao}
\date{}

\newcommand{\finalauthoraddress}[3]{%
  \par\bigskip
  \begingroup
  \hyphenpenalty=10000
  \exhyphenpenalty=10000
  \tolerance=3000
  \emergencystretch=2em
  \noindent{\normalfont\scshape (#1) #2\par}
  \endgroup
  \noindent\textit{Email address:}\ \texttt{#3}\par
}

\begin{document}
\maketitle

\begin{abstract}
Let $x$ be sufficiently large and
\[
N(x)=\big|\bigl\{n\le x:n\ \text{is odd and }n\ne p+2^a+2^b \textrm{ with } p 
\text{ a prime and } a,b\in \mathbb{N}\bigr\}\big|.
\]
Motivated by Crocker's result
\[
N(x)\gg \log\log x,
\]
Erd\H os repeatedly asked whether there is an absolute constant $c_0$ such that $N(x)>c_0x$. Pan \cite{Pan} proved in 2011 that
\[
N(x)\gg x\exp\!\left(
  -C_0\frac{\log\log\log\log x}{\log\log\log x}\log x
\right), 
\]
where $C_0>0$ is an absolute constant.
We improve on Pan's result by showing that, given any $\eta>0$, for all
sufficiently large $x$,
\[
  N(x)\gg_\eta x\exp\left(-(4+\eta)\frac{\log\log\log x}{\log\log x}\log x\right).
\]
\end{abstract}



\section{Introduction}

Unaware of the earlier correspondence between Euler and Goldbach \cite{Euler}, de Polignac \cite{Polignac1} conjectured in 1849 that every odd number greater than $3$ can be written as the sum of a prime and a power of two. Soon afterward, he \cite{Polignac2} realized that $127$ is a counterexample, and he further noted another counterexample, $959$, which had been pointed out by Euler. Motivated by this problem, Romanoff \cite{Romanoff} proved that the set of odd numbers which can be written as the sum of a prime and a power of two has positive lower asymptotic density. Conversely, van der Corput \cite{Corput} proved that the set of odd numbers which cannot be written as the sum of a prime and a power of two also has positive lower asymptotic density. Shortly thereafter, Erd\H os \cite{Erdos2} constructed an infinite arithmetic progression all of whose elements cannot be written as the sum of a prime and a power of two. In his construction, Erd\H os introduced covering congruence systems, which had a profound influence on subsequent research in combinatorial number theory.

Let $x$ be sufficiently large and
\[
N(x)=\big|\bigl\{n\le x:n\ \text{is odd and }n\ne p+2^a+2^b \textrm{ with } p \in \mathcal{P}\ \text{ and } a,b\in \mathbb{N}\bigr\}\big|,
\]
where $\mathcal{P}$ and $\mathbb{N}$ denote the sets of prime and natural numbers, respectively.
By an unusual refinement of Erd\H os' construction, Crocker \cite{Cr} proved that
\[
N(x)\gg \log\log x.
\]
Erd\H os \cite{Er77,Er80,EG,Er85,Er92,Er95,Er97-1,Er97-2} paid much attention to the size of $N(x)$, and repeatedly asked whether there is an absolute constant $c_0>0$ such that $N(x)>c_0x$. Erd\H os believed covering congruence systems no longer work here. Granville-Soundararajan \cite{GS} and Chen-Feng-Templier \cite{CFT} gave some conditional results on $N(x)$.  Answering a problem of Guy \cite[Line 1 of Page 43, A19]{Guy} in the affirmative, Pan \cite{Pan} proved
\[
N(x)\gg x\exp\left(-C_0\frac{\log\log\log\log x}{\log\log\log x}\log x\right),
\]
where $C_0>0$ is an absolute constant. This gives a lower bound of the
form $x^{1-o(1)}$. In particular, Pan's result implies
$N(x)>x^{1-\varepsilon}$ for any $\varepsilon>0$, provided that $x$ is
sufficiently large. On the other hand, Gallagher \cite{Gallagher} proved that
the density of odd integers which may be written in the form
\[
n=p+2^{a_1}+2^{a_2}+\cdots +2^{a_k} \qquad (p\in \mathcal{P},\ a_i\in \mathbb{N},\ 1\le i\le k)
\]
tends to $1$ as $k\to\infty$.

Our main result improves Pan's lower bound as follows.

\begin{theorem}\label{thm:uncond}
Let $\eta>0$ be a fixed number. Then for all sufficiently large $x$ we have
\[
  N(x)\gg_\eta x\exp\left(-(4+\eta)\frac{\log\log\log x}{\log\log x}\log x\right).
\]
\end{theorem}

Compared with Pan's estimate, Theorem~\ref{thm:uncond} replaces the factor
\(\log\log\log\log x/\log\log\log x\) in the exponent by the smaller factor
\(\log\log\log x/\log\log x\).  The new ingredient is a more flexible covering
construction.  We first use Fermat numbers to reduce to pairs with
$b-a=t2^m$.  For each remaining value of $t$, Stewart's theorem on large prime
divisors of cyclotomic values supplies primes $q_d\mid \Phi_d(2)$ with
$\ord_{q_d}(2)=d$ while $\sum_d1/q_d<\infty$.  Randomly chosen residue classes
modulo these $d$ cover almost all possible values of $a$, and the Chinese
remainder theorem then forces the associated prime $p$ into a small exceptional
set.  The few pairs not covered by this probabilistic construction are finally
handled by a Brun--Titchmarsh estimate.

\section{Auxiliary results}

The Fermat numbers
are
\[
  F_s=2^{2^s}+1\qquad(s=0,1,2,\ldots).
\]
They satisfy
\[
  \prod_{s=0}^{m-1}F_s=2^{2^m}-1.
\]

The first lemma records the elementary congruence that will remove many pairs.

\begin{lemma}\label{lem:fermatkill}
Let $s\ge0$, and let $a,b$ be nonnegative integers such that
$b-a=2^s u$ for some positive odd integer $u$. Then
\[
  2^a+2^b\equiv 0\pmod {F_s}.
\]
\end{lemma}
\begin{proof}
 Since $u$ is odd, we have
\[
2^{2^s u}+1=\big(2^{2^s}+1\big)\big(2^{2^s(u-1)}-2^{2^s(u-2)}+\cdots- 2^{2^s}+1\big),
\]
which implies $F_s\mid 2^{2^s u}+1$. If $b-a=2^s u$, then
\[
  2^a+2^b=2^a(1+2^{b-a})=2^a\big(1+2^{2^s u}\big),
\]
so the same divisibility gives the congruence.
\end{proof}

For a prime $q$ and an integer $a$ coprime to $q$, let $\ord_q(a)$ denote the
multiplicative order of $a$ modulo $q$.

The next lemma records the order of $2$ modulo prime divisors of Fermat
numbers.

\begin{lemma}\label{lem:fermatprime}
Let $s\ge0$, and let $\gamma_s$ be any prime divisor of $F_s$. Then
\begin{equation*}
\ord_{\gamma_s}(2)=2^{s+1}.
\end{equation*}
Moreover, for any choice of prime divisors $\gamma_s\mid F_s$,
\[
\sum_s\frac1{\gamma_s}<\infty.
\]
\end{lemma}

\begin{proof}
If $\gamma_s\mid F_s$, then $2^{2^s}\equiv-1\pmod{\gamma_s}$.  Therefore the
order of $2$ modulo $\gamma_s$ divides $2^{s+1}$ but does not divide $2^s$. It is
thus exactly $2^{s+1}$.  By Fermat's little theorem,
$2^{\gamma_s-1}\equiv1\pmod{\gamma_s}$, so the order of $2$ modulo
$\gamma_s$ divides $\gamma_s-1$. Hence
$2^{s+1}\mid \gamma_s-1$, and in particular $\gamma_s\ge 2^{s+1}+1$.
It follows that
\[
\sum_{s}\frac1{\gamma_s}\le \sum_{s}\frac1{2^{s+1}+1}<\infty,
\]
completing the proof of the lemma.
\end{proof}

We write $\Phi_d(X)$ for the $d$-th cyclotomic polynomial defined as
\[
\Phi_d(X)=\prod_{\substack{1\le k\le d\\ (k,d)=1}}\left(X-e^{\frac{2\pi i k}{d}}\right).
\]
It is standard that $\Phi_d(X)\in\mathbb{Z}[X]$, and by definition
$\Phi_d(X)\mid X^d-1$. See, for example, \cite[Section 1.5]{Murty}.

Let $P^+(n)$ be the largest prime factor of $n$ with convention $P^+(1)=1$.
For any $d\ge 3$, let
\begin{align}\label{eq-qd}
q_d=P^+\big(\Phi_d(2)\big).
\end{align}
The following remarkable result of Stewart lies at the heart of the improvement. 
\begin{lemma}\label{lem:stewart}
There is an absolute constant $d_0$ such that
\[
  \sum_{d\ge d_0}\frac1{q_d}<\infty,
\]
and
\(
  \ord_{q_d}(2)=d
\)
for $d>d_0$. 
\end{lemma}

\begin{proof}
The specialization $\alpha=2$, $\beta=1$ of
\cite[Theorem 1, pp.~293--294]{Stewart2013} gives constants $d_0$ and $c_0>0$
such that, for all $d\ge d_0$,
\[
  q_d>d\exp\left(c_0\frac{\log d}{\log\log d}\right).
\]
In Stewart's homogeneous notation, $\Phi_d(2,1)=\Phi_d(2)$.  Thus
\[
  \frac1{q_d}
  \le
  \frac1d\exp\left(-c_0\frac{\log d}{\log\log d}\right),
\]
implying the convergence of the infinite sum.

Because $\Phi_d(2)$ divides $2^d-1$, we have $2^d\equiv1\pmod {q_d}$.  A
standard property of cyclotomic values says that if a prime $q$ divides
$\Phi_d(2)$ and $q\nmid d$, then $\ord_q(2)=d$ (see e.g., \cite[Exercise 1.5.29]{Murty}).  Here $q_d>d$ for $d>d_0$, so certainly
$q_d\nmid d$.  Hence $\ord_{q_d}(2)=d$.

This completes the proof of the lemma.
\end{proof}

We will also use the following particular case of the Brun-Titchmarsh inequality. See, for example, Halberstam and Richert \cite[Theorem 3.7]{HalberstamRichert}.

\begin{lemma}\label{lem:sieve}
There is an absolute constant $C_1$ satisfying the following estimates. Let $W<x^{1/3}$
be odd, let $\beta\pmod {W}$ be fixed, and let $c$ be an integer with
\[
\gcd(\beta-c,W)=1.
\]
Then
\[
  \#\big\{n\le x:n\equiv\beta\pmodtight {W},\ n-c \ \textrm{is a prime} \big\}
  \le
  C_1\frac{x}{\varphi(W)\log x}.
\]
Consequently, after enlarging $C_1$ if necessary, if
\[
  \calS=\big\{n\le x:n\equiv\beta\pmodtight {W}\big\},
\]
then
\[
  \#\big\{n\in\calS:n-c \ \textrm{is a prime} \big\}
  \le
  C_1\frac{|\calS|}{\log x}\frac{W}{\varphi(W)}.
\]
\end{lemma}

\section{Proof of the theorem}

We will choose a suitable modulus $W$ and residue class $\beta$ such that a positive proportion of
\[
  \calS=\big\{n\le x:n\equiv\beta\pmodtight {2W}\big\}
\]
are not representable as $p+2^a+2^b$, where $p\in \mathcal{P}$ and $a\,b\in \mathbb{N}$.

We first introduce some parameters. Fix $\eta>0$.  Let
\[
  H=\left\lfloor\frac{\log x}{\log2}\right\rfloor,
  \qquad
  K=\left\lfloor \frac{\log\log x}{(4+\eta/2)\log\log\log x}\right\rfloor.
\]
Choose $m$ so that
\[
  K\le \frac{H}{2^m}<2K.
\]

Let $A$ be a large absolute constant, which will be chosen later.  Let
\[
  L=\log(AK).
\]
Choose $\delta>0$ small enough, depending only on $\eta$, so that
\[
  4+2\delta<4+\eta/2.
\]
Now define $M$ by
\[
  \log M=(2+\delta)K\log K.
\]

Write
\[
  \calD^*=\{d_0\le d\le M: d\text{ is not a power of }2\}.
\]
For large $x$,
\[
  \sum_{d\in\calD^*}\frac1d=\log M+O(1)\ge (2+\delta/2)K\log K.
\]
Since $L=\log(AK)=\log K+O_A(1)$, we have
\[
  (2K+1)L\le 2K\log K+O_A(K).
\]
A greedy construction now gives disjoint subsets $\calD_t\subseteq \calD^*$,
indexed by the integers $0\le t\le 2K$, such that
\begin{align}\label{random}
  \sum_{d\in\calD_t}\frac1d\ge L
  \qquad(0\le t\le 2K).
\end{align}

For every $d\in \calD^*$, take the prime $q_d$ defined by \eqref{eq-qd}.  For any $0\le s<m$, choose one prime factor $\gamma_s$ of $F_s$, i.e.,
\(
  \gamma_s\mid 2^{2^s}+1.
\)

\begin{lemma}\label{lem:disinct} The primes $2$, $\gamma_s~(0\le s<m)$ and $q_d ~(d\in \calD^*)$ are pairwise distinct.
\end{lemma}
\begin{proof}The primes $\gamma_s~(0\le s<m)$ are odd and pairwise distinct because Fermat numbers are odd and pairwise coprime.  

Note that $q_d$ is odd since $q_d\mid \Phi_d(2)$ and $\Phi_d(2)\mid 2^d-1$. Moreover, $q_d ~(d\in \calD^*)$ are pairwise distinct, since $q_{d_1}=q_{d_2}$ would imply $d_1=\ord_{q_{d_1}}(2)=\ord_{q_{d_2}}(2)=d_2$ by Lemma \ref{lem:stewart}.  

Finally, suppose that $q_d=\gamma_s$ for some $d\in \calD^*$ and $0\le s<m$. Then Lemma~\ref{lem:stewart} and Lemma~\ref{lem:fermatprime} give
\[
  d=\ord_{q_d}(2)=\ord_{\gamma_s}(2)=2^{s+1},
\]
which is a contradiction since the element in $\calD^*$ is not a power of $2$. Thus, $q_d\not=\gamma_s$ for any $d\in \calD^*$ and $0\le s<m$.
\end{proof}

Now we introduce 
\[
  W_0=\prod_{0\le s<m}\gamma_s, \ \ \ \ \ W_1=\prod_{0\le t\le 2K}\prod_{d\in\calD_t}q_d
\]
and
\[
W=2W_0W_1.
\]

\begin{proof}[Proof of Theorem~\ref{thm:uncond}] Throughout, we assume that $x$ is sufficiently large.
Fix $t$ and for any $d\in\calD_t$, choose a residue
$r_{t,d}\pmod d$ independently and uniformly at random.  All probabilities
and expectations below are taken over these random residues.  For fixed
$a\in[0,H]$, let $U_{a,t}$ denote the event that $a$ is uncovered for this
value of $t$, that is,
\[
  U_{a,t}=\big\{a\not\equiv r_{t,d}\pmod d
  \text{ for all }d\in\calD_t\big\}.
\]
Then, by \eqref{random},
\begin{align*}
  \mathbb{P}(U_{a,t})
  =
  \prod_{d\in\calD_t}\left(1-\frac1d\right)  
  \le
  \exp\bigg(-\sum_{d\in\calD_t}\frac1d\bigg)
  \le e^{-L}
  =
  \frac1{AK}.
\end{align*}
Let
\[
  \mathbf X_t=\sum_{0\le a\le H}\mathbf 1_{U_{a,t}}.
\]
Then, by linearity of expectation,
\[
  \mathbb{E}\mathbf X_t
  =
  \sum_{0\le a\le H}\mathbb{P}(U_{a,t})
  \le
  \sum_{0\le a\le H}\frac1{AK}
  =
  \frac{H+1}{AK}.
\]
Let $\mathbf X=\sum_{0\le t\le 2K}\mathbf X_t$ be the total number of uncovered pairs $(a,t)$.
Then
\[
  \mathbb{E}\mathbf X
  \le
  \frac{(2K+1)(H+1)}{AK}
  \le \frac{3(H+1)}{A}.
\]
We say $(a,t)$ is an uncovered pair if $a\not\equiv r_{t,d}\pmod d
  \text{ for all }d\in\calD_t$. Therefore, there is a deterministic choice of all
residues $r_{t,d}$ for which the total number $R$ of uncovered pairs satisfies
\begin{align}\label{bounduncovered}
  R \le \frac{3(H+1)}{A}\le \frac{4H}{A}.
\end{align}
We now fix such a deterministic choice of residues $r_{t,d}\pmod d$ for all
$0\le t\le 2K$ and $d\in\calD_t$.

Now we are able to introduce a congruence $\beta\pmod{W}$ as follows. 
In view Lemma \ref{lem:disinct}, one can choose a residue class $\beta\pmod {W}$ such that
\begin{align*}
\begin{cases}
  \beta\equiv1\pmod 2,
  \\ \beta \equiv 0\pmod{W_0},
  \\ \beta\equiv 2^{r_{t,d}}\big(1+2^{t2^m}\big)\pmod{q_d}\ \textrm{ for all }\ d\in \calD_t\ \textrm{ and }\ 0\le t\le 2K.
  \end{cases}
\end{align*}

Now we introduce \[
  \calS=\big\{n\le x:n\equiv\beta\pmodtight {W}\big\}
\]
as in Lemma \ref{lem:sieve}.

Let
\[
  \calE_1=\big\{n\in\calS:n=p+2^a+2^b\text{ for some prime }p\mid W
  \text{ and }0\le a,b\le H\big\}
\]
and 
\[
  \calE_2=\big\{n\in\calS\setminus \calE_1:n=p+2^a+2^b\text{ for some prime }p\nmid W
  \text{ and }0\le a,b\le H\big\}.
\]

Before handling $\calE_1$ and $\calE_2$, we first estimate the size of $W$.  Note that
\[
  \log W_0
  \le
  \log \Big(\prod_{s<m} F_s\Big)
  =
  \log\left(2^{2^m}-1\right)
  <2^m\log2
  \le
  \frac{H\log2}{K}
  \le
  \frac{\log x}{K}.
\]
By the definition of $K$, 
\begin{align}\label{boundW0}
  \log W_0
  \le
  \big(4+2\eta/3\big)\frac{\log\log\log x}{\log\log x}\log x.
\end{align}
Next, since $
\log q_d\le \log \Phi_d(2)<\log 2^d< d$, we have
\[
  \log W_1 \le \sum_{d\in \calD^\ast}q_d
  \le
  \sum_{d\le M}d
  \le M^2.
\]
Moreover,
\[
  2\log M=2(2+\delta)K\log K.
\]
Because
\[
  K\log K=
  \left(\frac{1}{4+\eta/2}+o(1)\right)\log\log x,
\]
our choice of $\delta$ gives
\[
  2\log M\le (1-c_\eta)\log\log x
\]
for some $c_\eta>0$.  Therefore
\[
  M^2\le (\log x)^{1-c_\eta},
\]
and hence
\begin{align}\label{boundW1}
  \log W_1=o\left(\frac{\log\log\log x}{\log\log x}\log x\right).
\end{align}
Combining \eqref{boundW0} and \eqref{boundW1}, we obtain
\begin{align}\label{eq-logW-final}
  \log W=\log 2+\log W_0+\log W_1 &\le 
  \big(4+2\eta/3+o(1)\big)\frac{\log\log\log x}{\log\log x}\log x\nonumber\\
  &<\big(4+\eta\big)\frac{\log\log\log x}{\log\log x}\log x.
\end{align}
In particular, $\log W=o(\log x)$, so
$W<x^{1/3}$.  Additionally,
\begin{align}\label{boundS}
  |\calS|=\frac{x}{W}+O(1).
\end{align}
The above estimates for $W$ and $\calS$  give
\begin{align}\label{eq-negligible}
  |\calE_1|\le (H+1)^2\omega(W)
  \le \left(\frac{\log x}{\log 2}+1\right)^2\frac{\log W}{\log 2}
  \ll (\log x)^3=o(|\calS|).
\end{align}

Take $n\in\calE_2$ and we can express $n$ as
\[
  n=p+2^a+2^b,\qquad p\nmid W
  \qquad 0\le a\le b\le H.
\]
If
$b-a\not\equiv0\pmod {2^m}$, then $b-a=2^su$ for some $s<m$ and $2\nmid u$.  By
Lemma~\ref{lem:fermatkill},
\[
  2^a+2^b\equiv0\pmod {\gamma_s}.
\]
For $n\in \calS$, we have $n\equiv\beta\equiv0\pmod {\gamma_s}$.  Hence, we obtain
\[
  p=n-2^a-2^b\equiv0\pmod {\gamma_s}.
\]
This implies $p=\gamma_s\mid W$ and $n\in \calE_1$, a contradiction to $n\in \calE_2$. Thus, we proved that $b-a\equiv0\pmod {2^m}$. Then we can suppose that
\begin{align*}
  b=a+t2^m,
  \qquad
  0\le t\le 2K.
\end{align*}

We first claim that $(a,t)$ is an uncovered pair. Suppose otherwise, for some $d\in\calD_t$,
\[
  a\equiv r_{t,d}\pmod d.
\]
Since $\ord_{q_d}(2)=d$, we have
\[
  2^a\equiv 2^{r_{t,d}}\pmod {q_d}.
\]
Then we deduce that
\[
  2^a+2^b
  =2^a(1+2^{t2^m})
  \equiv
  2^{r_{t,d}}(1+2^{t2^m})
  \equiv \beta
  \pmod{q_d}.
\]
For $n\in \calS$, we have $n\equiv\beta\pmod {q_d}$. Then we obtain
\[
  p=n-2^a-2^b\equiv0\pmod {q_d}.
\]
This implies $p=q_d\mid W$ and $n\in \calE_1$, a contradiction, and thus the above claim is proved. 

We also claim $(\beta-2^a-2^{a+t2^m},W)=1$. Suppose otherwise  $(\beta-2^a-2^{a+t2^m},W)$ has a prime divisor $\ell$. We deduce that
$p=n-2^a-2^{a+t2^m}\equiv \beta-2^a-2^{a+t2^m}\equiv 0\pmod{\ell}$. Then, $p=\ell$ and a contradiction to $n\in \calE_2$. Thus, $(\beta-2^a-2^{a+t2^m},W)=1$ is true.

Now we conclude from the above claims that 
$$|\calE_2|\le \sum_{(a,t)\ \textrm{is an uncovered pair}}|\{n\in \calS:\ n-2^a-2^{a+t2^m} \ \textrm{is a prime} \}|,$$
and by Lemma~\ref{lem:sieve},
\[
  \big|\big\{n\in\calS:n-2^a-2^{a+t2^m} \ \textrm{is a prime}\big\}\big|
  \le
  C_1\frac{|\calS|}{\log x}\frac{W}{\varphi(W)}.
\]

By \eqref{bounduncovered}, there are at most $\frac{4H}{A}$ uncovered pairs. We obtain
 \begin{align}\label{boundE2first}|\calE_2|\le 
  \frac{4C_1}{A}\cdot \frac{H}{\log x}\cdot |\calS|\cdot\frac{W}{\varphi(W)}.\end{align}
It remains to record the estimate for $W/\varphi(W)$.  By Lemmas~\ref{lem:fermatprime} and~\ref{lem:stewart},
\[
  \sum_{\substack{q\mid W}}\frac1q
  =\frac{1}{2}+
  \sum_{s<m}\frac1{\gamma_s}
  +\sum_{\substack{q\mid W_1}}\frac1q
  \ll 1,
\]
where the letter $q$ denotes a prime,  
and thus
\begin{align}\label{eq-W}
  \frac{W}{\varphi(W)}
  =
  \prod_{\substack{q\mid W}}\left(1-\frac1q\right)^{-1}
  \ll 1.
\end{align}
Recalling $H\le \log_2x$, we conclude from \eqref{boundE2first} and \eqref{eq-W} that 
 \begin{align*}|\calE_2|\le 
  \frac{C_2}{A}\cdot  |\calS|\end{align*}
  for some absolute constant $C_2$ (independent of $A$). 
Choosing the constant $A=4C_2$, we obtain 
 \begin{align}\label{boundE2}\calE_2\le |\calS|/4.\end{align}
 
 Now we obtain from \eqref{boundS}, \eqref{eq-negligible} and \eqref{boundE2} that
\[
  N(x)\ge |\calS|-|\calE_1|-|\calE_2|\ge |\calS|/2\gg \frac{x}{W}.
\]
This proves Theorem~\ref{thm:uncond} in view of \eqref{eq-logW-final}.
\end{proof}

\section*{Acknowledgments}
This work is supported by the National Key Research and Development Program of China (Grant No. 2021YFA1000700) and the National Natural Science Foundation of China (Grant No. 12471088).

\finalauthoraddress{Yuchen Ding}{School of Mathematical Sciences, Yangzhou University, Yangzhou 225002, People's Republic of China\\
HUN-REN Alfr\'ed R\'enyi Institute of Mathematics, Budapest, Pf. 127, H-1364 Hungary}{ycding@yzu.edu.cn}

\finalauthoraddress{Yu-Chen Sun}{School of Mathematics, University of Bristol, Bristol, BS8 1UG, England}{yuchensun93@163.com}

\finalauthoraddress{Lilu Zhao}{School of Mathematical Science, University of Science and Technology of China, Hefei 250100, People's Republic of China}{zhaolilu@ustc.edu.cn}

\begin{thebibliography}{99}
\small
\setlength{\itemsep}{1pt plus 0.2pt}
\setlength{\parsep}{0pt}

\bibitem{CFT} 
Y.-G. Chen, R. Feng and N. Templier, {\it Fermat numbers and integers of the form $a^k+a^l+p$,} Acta Arith. {\bf135} (2008), 51--61.

\bibitem{Corput} 
J. G. van der Corput, {\it On de Polignac's conjecture,} Simon Stevin {\bf 27} (1950), 99--105.

\bibitem{Cr} 
R. Crocker, {\it On the sum of a prime and two powers of two,} Pacific J. Math. {\bf 36} (1971), 103--107.

\bibitem{Erdos2} 
P. Erd\H os, {\it On the integers of the form $2^k+p$ and some related problems,} Summa Brasil. Math. {\bf 2} (1950), 113--123.

\bibitem{Er77} 
P. Erd\H os, {\it  Problems and results on combinatorial number theory. III.}  Number theory day (Proc. Conf., Rockefeller Univ. New York, 1976) (1977), 43--72.

\bibitem{Er80} 
P. Erd\H os, {\it A survey of problems in combinatorial number theory,} Ann. Discrete Math. {\bf 6} (1980), 89--115.

\bibitem{EG} 
P. Erd\H os and R. Graham, {\it Old and new problems and results in combinatorial number theory,} Monographies de L'Enseignement Math\'ematique (1980). 

\bibitem{Er85} 
P. Erd\H os, {\it On some of my problems in number theory I would most like to see solved,} Number theory (Ootacamund, 1984) (1985), 74--84.

\bibitem{Er92} 
P. Erd\H os, {\it Some of my forgotten problems in number theory,}
Hardy-Ramanujan J. {\bf15} (1992), 34--50.

\bibitem{Er95} 
P. Erd\H os, {\it Some of my favourite problems in number theory, combinatorics, and geometry,}
Resenhas {\bf2} (1995), 165--186.

\bibitem{Er97-1} 
P. Erd\H os, {\it Problems in number theory,}
New Zealand J. Math. {\bf26} (1997), 155--160.

\bibitem{Er97-2} 
P. Erd\H os, {\it Some of my favourite unsolved problems,}
Math. Japan. {\bf 46} (1997), 527--537.

\bibitem{Euler} 
L. Euler, {\it Letter to Goldbach, 16.12.1752.} \\
http://eulerarchive.maa.org/correspondence/correspondents/Goldbach.html

\bibitem{Gallagher} 
P. X. Gallagher, {\it Primes and powers of 2,} Invent. Math. {\bf 29} (1975), 125--142.

\bibitem{GS}
A. Granville and K. Soundararajan, {\it A binary additive problem of Erd\H os and the order of $2$ mod $p^2$,} Ramanujan J. {\bf2} (1998), 283--298.

\bibitem{Guy} R. K. Guy, {\it Unsolved Problems in Number Theory,} 2nd edition, Springer, New York, 1994.


\bibitem{Murty}
M. R. Murty, {\it Problems in Analytic Number Theory,} 2nd edition, Grad. Texts in Math. 206, Springer, New York, 2008.


\bibitem{HalberstamRichert}
H. Halberstam and H.-E. Richert,
\newblock \emph{Sieve Methods},
\newblock London Mathematical Society Monographs, No. 4, Academic Press, London, 1974.

\bibitem{Pan} 
H. Pan, {\it On the integers not of the form $p+2^a+2^b$,} Acta Arith. {\bf 148} (2011), 55--61.

\bibitem{Polignac1} 
A. de Polignac, {\it Six propositions arithmologiques d\'eduites du crible d'Eratosth\`ene,}  Nouv. Ann. Math. {\bf 8} (1849), 423--429.

\bibitem{Polignac2} 
A. de Polignac, {\it Rectification \`a une pr\'ec\'edente communication relative \`a la th\'eorie des nombres,} C. R. Acad. Sci. Paris {\bf29} (1849), 738--739.

\bibitem{Romanoff} 
N. P. Romanoff, {\it \"Uber einige S\"atze der additiven Zahlentheorie,} Math. Ann. {\bf 109} (1934), 668--678.

\bibitem{Stewart2013}
C. L. Stewart,
\newblock On divisors of Lucas and Lehmer numbers,
\newblock \emph{Acta Math.} \textbf{211} (2013), no.~2, 291--314.


\end{thebibliography}
\end{document}